\documentclass[UTF-8,reqno,12pt]{amsart}
\usepackage{enumerate}
\usepackage{mhequ}
\usepackage{stmaryrd}
\usepackage{bbm}
\usepackage{geometry}

\usepackage{footnote}
\usepackage[pagewise]{lineno}%\linenumbers

\usepackage{amssymb,url,color, booktabs,nccmath}
\usepackage{mathrsfs}
\usepackage{graphicx}
\usepackage{tikz}
\usepackage{amsthm}

\usepackage{relsize}%强制拉直符号, 使用\mathlarger{\int}

\usepackage{color}
\usepackage{hyperref}
\hypersetup{
	colorlinks=true,       % false: boxed links; true: colored links
	linkcolor=blue,          % color of internal links
	citecolor=red,        % color of links to bibliography
	filecolor=blue,      % color of file links
	urlcolor=cyan
}

\definecolor{darkergreen}{rgb}{0.0, 0.5, 0.0}

\makeatletter
		\renewcommand{\subsection}{\@startsection
			{subsection} % name
			{2} % level
			{0mm} % indent
			{0.5\baselineskip} % befores
			{0.3\baselineskip} % afterskip
			{\normalfont\normalsize\raggedright}} % style
\makeatother
  
\allowdisplaybreaks[4]

\usepackage{fancyhdr} % 用于自定义页眉
\numberwithin{equation}{section}

\newcommand{\be}{\begin{eqnarray}}
	\newcommand{\ee}{\end{eqnarray}}
\newcommand{\ce}{\begin{eqnarray*}}
	\newcommand{\de}{\end{eqnarray*}}
\newtheorem{theorem}{Theorem}[section]
\newtheorem{lemma}[theorem]{Lemma}
\newtheorem{remark}[theorem]{Remark}
\newtheorem{definition}[theorem]{Definition}

\newtheorem*{theorem*}{Theorem}
\newtheorem*{remark*}{Remark}

\newtheorem{conditionA}{Assumption A\kern-0.1mm}
\labelformat{conditionA}{$\mathbf{A\kern-0.1mm#1}$}

\newtheorem{conditionC}{Condition C\kern-0.1mm}
\labelformat{conditionC}{$\mathbf{C\kern-0.1mm#1}$}

\def\ge{\geqslant}
\def\le{\leqslant}

\allowdisplaybreaks

\def\E{\mathbb{E}}
\def\P{\mathbb{P}}
\def\R{\mathbb{R}}

\def\d{\mathrm{d}}
\def\dif{\mathrm{d}}

\def\1{\mathbbm{1}}
\def\e{\mathrm{e}}

\def\geq{\geqslant}
\def\leq{\leqslant}

\begin{document}

	\title{
Ergodicity of stochastic functional differential equation with jumps and finite delay}

	\author{Mingkun Ye$^1$, Yafei Zhai$^{2,*}$ and Zuozheng Zhang$^1$}
	% \address[]{Department of Mathematics, Beijing Institute of Technology, Beijing 100081, P.R.China; School of Mathematics and Statistics, Wuhan University, Wuhan, Hubei 430072, P.R.China
	% }
	% \email{mingkunye@foxmail.com(M. Ye)\\ zuozhengzhang@mail.bnu.edu.cn(Z. Zhang)\\ yafeizhai@bit.edu.cn}	
\thanks{$^{*}$ Corresponding author.}
\dedicatory{$^1$ School of Mathematics and Statistics, Wuhan University, Wuhan, 430072, P.R.China\\$^2$
		School of Mathematics and Statistics,
		Beijing Institute of Technology, 
		Beijing, 100081,  P.R.China}
\thanks{E-mail addresses: mingkunye@foxmail.com(M. Ye), yafeizhai@bit.edu.cn(Y. Zhai),\\ zuozhengzhang@mail.bnu.edu.cn(Z. Zhang)}	
    
	\begin{abstract}
    This paper investigates the ergodicity of stochastic functional differential equations with jumps under the Wasserstein distance by the generalized coupling method.  Two key conditions are verified. The first is verified by establishing an exponential decay bound for the coupled segment processes and applying the Girsanov theorem for It\^o-L\'evy processes. The second  is verified through a support theorem developed for an auxiliary process and then extended to the underlying process. Combining these results yields the desired ergodicity.
    \\
	% 摘要后面紧跟的 AMS数学分类
	%{\it AMS Mathematics Subject Classification :} 60J25; 60J60; 60J76.	
	\end{abstract}

	%\keywords{} 这样关键词在页脚处
	
	%\date{\today}
	
	\maketitle

 %在生成的目录中，只显示到节级别，即章节和节，而不会显示子节和子子节
\setcounter{tocdepth}{1}
 
 %生成目录
%\tableofcontents

\section{Introduction}

Stochastic functional differential equations (SFDEs) have emerged as a significant field at the intersection of stochastic analysis and dynamical systems theory. These equations provide a natural mathematical framework for systems whose evolution is delineated not only by instantaneous stochastic inputs but also by a continuous dependence on their past states, thereby formally capturing dynamic phenomena with inherent memory or hysteresis. There are many known results concerning this model; see  \cite{bao2014,butkovsky2017,chen2024,li2019,ma2025,mao2007, mohammed1984,wu2017stochastic} for more details.
% Classical theory, under linear growth and Lipschitz continuity conditions, has established fundamental results such as existence, uniqueness, stability, boundedness, and ergodicity of solutions (see, e.g., \cite{mao2007, mohammed1984}). Moreover, Butkovsky and Scheutzow \cite{butkovsky2017} prove the exponential or subexponential convergence holds under a one-sided Lipschitz condition. Bao et.~al. \cite{bao2014} also investigate ergodicity using the remote start or dissipative method.  Ma et.~al. \cite{ma2025} consider numerical approximation of stochastic differential delay equations. 
It is notable that the aforementioned studies on SFDEs    are primarily driven by Brownian motion.

However, when such perturbations are known to exhibit extreme behavior or discontinuous behavior, models based on Brownian motion become inadequate. To capture such jump behaviors, a suitable alternative is provided by models driven by L\'evy noise. In recent years, the ergodicity of stochastic differential equations with Lévy noise has been extensively studied using various probabilistic and analytic methods.
A well‑established approach is rooted in the Meyn-Tweedie stability theory, where Lyapunov functions are employed to formulate coefficient conditions, combined with regularity assumptions such as the strong Feller property and irreducibility, to establish exponential ergodicity. This framework is quite versatile: it applies under rather mild conditions on the coefficients and yields convergence in different metrics, including total variation and Wasserstein distance.
% A classical result in this direction is due to Xie and Zhang [1]. In which, for equations with singular coefficients, they employed transformation techniques such as the Zvonkin transform and PDE methods to regularize the system, and then used Lyapunov conditions together with Krylov‑type estimates to verify the general criterion for exponential ergodicity established by Goldys and Maslowski in \cite[Theorem 2.5]{Goldys2006}. 
A classical result pertaining to this research direction is established by Xie and Zhang \cite{XIEandZHANG2020}, who utilized transformation techniques including the Zvonkin transform and PDE-based methods for the regularization of systems with singular coefficients. By combining Lyapunov conditions with Krylov-type estimates, they validate the general criterion for the exponential ergodicity proposed by Goldys and Maslowski in \cite[Theorem 2.5]{Goldys2006}. However, convergence rates obtained via such methods are usually not easy to quantify. 

Classical coupling methods, which construct a pair of stochastic processes with strictly prescribed marginal laws, have long served as a fundamental tool in establishing the ergodicity property of Markov processes; see \cite{chen2004,chen1989,lindvall1986,priola2006} for example. However, the stringent requirement that the marginal laws of the coupled processes must exactly coincide with a given pair of probability distributions poses significant technical challenges, especially when dealing with systems driven by degenerate or non-elliptic noise, infinite-dimensional stochastic partial differential equations, or diffusion processes on non-compact state spaces. To overcome these limitations, the concept of asymptotic coupling \cite{hairer2002exponential,Hairer2011PTRF,mattingly2002exponential} is introduced in the spirit of earlier works, allowing the marginal laws of the coupled processes to merely approximate, rather than exactly match a prescribed pair of probability distributions. This relaxation, when combined with classical results such as the Birkhoff ergodic theorem, proved to be a powerful mechanism for establishing unique ergodicity. This notion was subsequently extended and refined under the name generalized coupling \cite{Butkovsky2020,glatt2017unique,kulik2018generalized}, motivated by the observation that the underlying idea could be deployed in a nonasymptotic fashion to guarantee weak stabilization of transition probabilities, thereby providing a unified and flexible framework for proving both the uniqueness of stationary measures and quantitative or qualitative weak convergence to the invariant probability measure across a broad class of stochastic models.

In this paper, we aim to generalize the results of Hairer et al. \cite{Hairer2011PTRF} to SFDEs  with finite delay and driven simultaneously by Brownian motion and jump processes. Let \( (\Omega, \mathcal{F}, \mathbb{P}) \) be a complete probability space equipped with some filtration \( \left(\mathcal{F}_{t}\right)_{t \geqslant 0} \) satisfying the usual conditions (i.e., it is right continuous and \( \mathcal{F}_{0} \) contains all \( \mathbb{P} \)-null sets).  For any given $0< \tau<\infty$, let \( \mathscr{D}\) denote the family of all right-continuous functions with left-hand limits from   \( [-\tau, 0]\) to \(\mathbb{R}^{n}. \)
Then, let $\Lambda$ be the set of all continuous functions  $\lambda: [-\tau, 0] \rightarrow [-\tau, 0]$  that are strictly increasing, with $\lambda(0)=0$ and  $\lambda(-\tau) =-\tau$. For all \(\lambda \in \Lambda\), we set
\[
  \interleave\lambda \interleave = \sup_{-\tau\leq s < t\leq 0} \left| \log \frac{\lambda(t) - \lambda(s)}{t - s} \right|. 
\]
 For \(\phi, \psi\in \mathscr{D}\), let $\|\phi\|_{\infty}:=\sup_{-\tau\leq \theta\leq 0}|\phi(\theta)|$, and
\[
 d(\phi, \psi) :=  \mathlarger{\inf_{\lambda \in \Lambda} }\left(\interleave\lambda\interleave + \|\phi \circ \lambda -  \psi\|_{\infty} \right), 
\]
 then $(\mathscr{D},d)$ is a Polish space and moreover, $d(\phi, \psi) \leq \|\phi-\psi\|_{\infty}$. In fact, this can be obtained by setting \( \lambda(t) = t \). Under the uniform metric, the space \( \mathscr{D} \) is complete but not separable. However, under the Skorohod metric \( d \), \( \mathscr{D} \) is not only complete but also separable. For further details on the Skorohod metric, we refer the reader to \cite[Chapter 4]{Billingsley1999BOOK}.

Consider the following stochastic functional differential equations  with jumps (SFDEwJ), 
\begin{equation}\label{EQ:0603:01}
    \d X(t) = b(X_t)\d t+\sigma(X_t) \d W(t)+\int_{Z}\gamma(X_{t-})c(z)\widetilde{N}(\d t,\d z),
\end{equation}
where $b:\mathscr{D}\to \mathbb{R}^n$, $\sigma:\mathscr{D}\to\mathbb{R}^n\otimes\mathbb{R}^m$ , $\gamma:\mathscr{D}\to \mathbb{R}^n$ and $c:Z\to \mathbb{R}^n$. $X_t(\theta):=X(t+\theta)$, $-\tau\leq \theta \leq 0$, is the segment process of $X(t)$. $X_{t-}(\theta):=d^{\prime}\text{-}\lim_{s\uparrow t}X_s(\theta)$ for any $\theta\in [-\tau,0]$, where $d^{\prime}\text{-}\lim$ is defined in the sense of the Euclidean distance.
$(Z, \mathcal{Z}):= (\R_0^{n},\mathscr{B}(\R_0^n))$ is a measurable space and $W(t)$ is an $\mathbb{R}^n$-valued Brownian motion. $N(\d t, \d z)$ (corresponding to a random point process $p(t)$) is a Poisson random measure independent of $W(t)$. $\nu(\dif z)$ is a deterministic finite characteristic measure on $(Z, \mathcal{Z})$ and $\widetilde{N}(\d t, \d z) = N(\d t, \d z) - \nu(\d z)\d t$ is the compensated Poisson random measure on $[0, \infty) \times Z$.

The purpose of this paper is to investigate the ergodicity of SFDEwJ \eqref{EQ:0603:01} with respect to the Wasserstein distance $W_{d \wedge 1}$. Our analysis is based on the generalized coupling method, a powerful tool for establishing convergence in optimal transport metrics; see Theorem \ref{Prop:0603:01} for more details.  It is sufficient to verify Conditions \ref{C1} and \ref{C2}.  This verification for Condition \ref{C1} consists of two sub-steps: 
\begin{enumerate}[(1)]
    \item To establish Condition \ref{C1}.2, we first obtain the exponential decay bound for the difference between the coupled process pair $(X(t), Y(t))$ and then get the corresponding result for the segment processes $(X_t, Y_t)$ using the BDG inequality for local martingales.%, which guarantees that Condition \ref{C1}.2 holds.
    \item Condition \ref{C1}.1 is verified by a Girsanov theorem for It\^o-L\'evy processes and the exponential decay bound for $\mathbb E|X(t)-Y(t)|^2$ in (1). 
\end{enumerate}
  To verify Condition \ref{C2}, we first establish the support theorem with respect to the auxiliary process \eqref{eq:aux1}, and then further derive the support theorem associated with SFDEwJ \eqref{EQ:0603:01}. Based on this, we conclude that Condition \ref{C2} is satisfied.

The rest of this paper is organized as follows. In Section \ref{sec:2}, we  present our main result. In Sections \ref{SEC:03} and \ref{SEC:04}, we verify Conditions \ref{C1} and  \ref{C2}, respectively.

\section{Main result}\label{sec:2}
In this section, we state our main result. The proofs will be given in subsequent sections. We first give some basic definitions. 
\begin{definition}
    A stochastically continuous Markovian semigroup $P_t$ is called eventually Feller (see \cite{REISS2006SPA}) if there exists a $t_0 \geq 0$ such that, for any $h \in C_b(\mathscr{D})$, the following two conditions are satisfied:
\begin{equation}\label{eq:0427i}
    P_t h \in C_b(\mathscr{D} ), \quad \text{for every } t \geq t_0, 
\end{equation}
\begin{equation}\label{eq:0427it}
    \lim_{s \downarrow t} P_s h(\xi) = P_t h(\xi), \quad \text{for every } \xi \in \mathscr{D} \text{ and } t \geq t_0.
\end{equation}
In particular, when $t_0 = 0$, $P_t$ is Feller. 
\end{definition}

Define the  $L^1$-Wasserstein (or Kantorovich) distance between two probability measures  $\mu, \nu \in   \mathscr{P}(\mathscr{D})$ as follows:   
	$$
	W_d(\mu, \nu) = \inf _{\pi \in \Pi(\mu, \nu)} \int_{\mathscr{D}\times \mathscr{D}} d(\phi,\psi)\pi(\mathrm{d}\phi,\mathrm{d}\psi),
	$$
	where  $\Pi(\mu, \nu)$  is the collection of probability measures on $\mathscr{D}\times \mathscr{D}$ having $\mu$ and $\nu$ as marginals. We need the following pair of notions.
\begin{definition}
     A distance-like function $d$ bounded by $1$ is called contracting for $P_{t}$ if there exists $\alpha<1$ such that for any $\xi,\eta \in \mathscr{D}$ with $ d(\xi,\eta)<1$, we have
\begin{equation*}
    W_{d}\left(P_{t}(\xi, \cdot), P_{t}(\eta, \cdot)\right) \leq \alpha d(\xi, \eta),
\end{equation*}
where $P_t(\xi,\cdot)$ and $P_t(\eta,\cdot)$ are the transition functions of the process $(X_{t}^{\xi})_{t \geq 0}$ and $(X_{t}^{\eta})_{t \geq 0}$, respectively.
\end{definition}

\begin{definition}
    A set $B \subset \mathscr{D}$ is called $d$-small for $P_{t}$ if for some $\varepsilon>0$,
\begin{equation*}
    \sup _{\xi,\eta \in B} W_{d}\left(P_{t}(\xi, \cdot), P_{t}(\eta, \cdot)\right) \leq 1-\varepsilon.
\end{equation*}
\end{definition}

The following condition \ref{C1} will serve as a replacement for the contractivity condition.

\begin{conditionC}\label{C1}
     There exist a non-increasing function $r: \mathbb{R}_{+} \rightarrow \mathbb{R}_{+}$ with $\lim _{t \rightarrow \infty} r(t)=0$ and a locally bounded function $C(t): \mathbb{R}_{+} \rightarrow \mathbb{R}_{+}$ such that for any $\xi,\eta \in \mathscr{D}$, there exist random processes $X^{\xi}=(X_{t}^{\xi})_{t \geq 0}$ and $Y^{\eta}=(Y_{t}^{\eta})_{t \geq 0}$ with the following properties:
\begin{enumerate}
    \item[1.]  %$\operatorname{Law}(X^{\xi})=\mathbb{P}_{\xi}$, and
    For any $  t \geq 0$, $d_{\mathrm{TV}}(\operatorname{Law}(Y_{t}^{\eta}), P_{t}(\eta, \cdot)) \leq C(t) \|\xi-\eta\|_{\infty},$
where $d_{\mathrm{TV}}$ denotes the total variation distance.
\item [2.]  $\mathbb{E} \|X_{t}^{\xi}- Y_{t}^{\eta}\|_{\infty} \leq r(t) \|\xi-\eta\|_{\infty}, \quad t \geq 0.$
\end{enumerate}
\end{conditionC}

Condition \ref{C2} will replace the $d$-small property.
\begin{conditionC}\label{C2}
    There exist a set $B \subset \mathscr{D}$ and $t_{0}>0$ such that for any $\varepsilon>0$, there exists a set $D \in \mathscr{B}(\mathscr{D})$ such that:
\begin{enumerate}
    \item [1.] $\inf _{\xi \in B} P_{t_{0}}(\xi, D)>0$;
\item [2.] $\sup _{\xi,\eta \in D} \|\xi-\eta\|_{\infty} \leq \varepsilon$.
\end{enumerate}
\end{conditionC}

\begin{theorem}[{\cite[Theorem 2.6]{Butkovsky2020}}]\label{Prop:0603:01}
 Assume that the Markov {semigroup $(P_t)$ is Feller}. Suppose that:
\begin{enumerate}
\item [1.] There exists a measurable function $V: \mathscr{D} \rightarrow[0, \infty)$ which satisfies the Lyapunov condition, that is, there exist a concave differentiable function $f: \mathbb{R}_{+} \rightarrow \mathbb{R}_{+}$ increasing to infinity with $f(0)=0$ and a constant $K>0$ such that for any $t \geq 0, \xi \in \mathscr{D}$,
\begin{equation*}
    P_{t} V(\xi) \leq V(\xi)-\int_{0}^{t} P_{s}(f \circ V)(\xi) \d s+K t.
\end{equation*}

\item[2.] For any $f,g\in\mathscr{D}$, $d(f,g)\wedge 1 \leq \|f-g\|_{\infty}$.

\item[3.] Condition \ref{C1} holds for functions $r, L$.

\item [4.]There exists $t_{0}>0$ such that for any $M>0$, Condition \ref{C2} holds for  $B=\{V \leq M\}$ and $t_{0}$.
\end{enumerate}
Then the Markov semigroup $(P_{t})$  has a unique invariant measure $\pi$. Moreover, for any $\delta \in(0,1)$, there exist constants $C_{1}, C_{2}>0$ such that for any $\xi \in \mathscr{D}$,
\begin{equation}\label{EQ:0603:05}
    W_{d \wedge 1}\left(P_{t}(\xi, \cdot), \pi\right) \leq \frac{C_{1}\left(1+f(V(\xi))^{\delta}\right)}{f\left(F^{-1}\left(C_{2} t\right)\right)^{\delta}}, \quad  t \geq 0,
\end{equation}
where $F(x)$ is defined by
\begin{equation}\label{EQ:0603:03}
    F(x):=\int_{1}^{x} \frac{1}{f(u)}\d u,  \quad x \geq 1.
\end{equation}
\end{theorem}

Then, we  propose  some necessary assumptions for the coefficients in SFDEwJ \eqref{EQ:0603:01}.
Suppose that the functions $b(\cdot)$, $\sigma(\cdot)$, and $\gamma(\cdot)$ are continuous with respect to Skorohod topology and bounded on bounded subsets of $\mathscr{D}$. Moreover, suppose that the function $c(\cdot)$ is integral with respect to the measure $\nu$.

\begin{conditionA}\label{A1}
    There exist a positive constant $K$ and a probability measure $\mu$ on $[-\tau,0]$ such that $\int_{Z}|c(z)|^2\nu(\d z)\leq K$, and for any $\phi,\psi\in \mathscr{D}$, 
\begin{equation*} 
    \begin{split}
         2&\langle \phi(0)-\psi(0), b(\phi)-b(\psi)\rangle_{+} +\|\sigma(\phi)-\sigma(\psi)\|_{\rm{HS}}^2+|\gamma(\phi_{-})-\gamma(\psi_{-})|^2\\
         &\leq K|\phi(0)-\psi(0)|^2+K\int_{-\tau}^0|\phi(\theta)-\psi(\theta)|\mu(\d \theta),
    \end{split}
\end{equation*}
where $\langle \cdot,\cdot\rangle_+:=\max{\{\langle \cdot,\cdot\rangle,0\}}$, $\|\cdot\|_{\mathrm{HS}}$ is the Hilbert-Schmidt operator norm for matrices in $\mathbb{R}^d\otimes \mathbb{R}^d$, and $\phi_{-}(\theta):=d^{\prime}\text{-}\lim_{\theta^{\prime}\uparrow \theta}\phi(\theta^{\prime})$ for any $\theta\in [-\tau,0]$.
% \footnote{\blue If $\nu$ is a finite measure, $\int_{Z}|c(\phi_{-},z)-c(\psi_{-},z)| \nu (\d z)\leq C (\int_{Z}|c(\phi_{-},z)-c(\psi_{-},z)|^2 \nu (\d z))^{1/2}$?}
\end{conditionA}
\begin{remark}
   For any process $(X(t))_{t\geq 0}$, and we claim that $X_{t-}=(X_t)_{-}$. Indeed, for any $\theta\in [-\tau,0]$,
    \begin{equation*}
        \begin{aligned}
            X_{t-}(\theta)&=d^{\prime}\text{-}\lim_{s\uparrow t}X_s(\theta)=d^{\prime}\text{-}\lim_{s\uparrow t}X(s+\theta)\\
            &=d^{\prime}\text{-}\lim_{\theta^{\prime}\uparrow \theta}X(t+\theta^{\prime})=d^{\prime}\text{-}\lim_{\theta^{\prime}\uparrow \theta}X_t(\theta^{\prime})=(X_t)_{-}(\theta).
        \end{aligned}
    \end{equation*}
\end{remark}
By virtue of Assumption \ref{A1} and an argument analogous to that in \cite{vonRenesseScheutzow}, SFDEwJ \eqref{EQ:0603:01} has a unique solution \((X^{\xi}(t))_{t \geq 0}\) with \(X_0 = \xi\). Moreover, by analogy with \cite{wu2017stochastic}, we can establish  that the process $(X_t)_{t\geq 0}$ is homogeneous and strong Markovian.  It follows from \cite[Proposition 3.5]{REISS2006SPA} that $P_t$ is eventually Feller, namely that \((X^{\xi}(t))_{t \geq \tau}\) possesses the Feller property after time $t = \tau$.
% \begin{conditionA}\label{A2}
%    For any $\phi,\psi\in\mathscr{D}$, we have
% \begin{equation*} 
%     \int_{Z}|c(\phi_{-},z)-c(\psi_{-},z)| \nu (\d z)
%         \leq K\left(|\phi(0)-\psi(0)|+\int_{-\infty}^{0}|\phi(\theta)-\psi(\theta)| \mu(\d \theta)\right).
% \end{equation*}
% \end{conditionA}

\begin{conditionA}\label{A3}
    For any $\phi\in\mathscr{D}$, $\sigma(\phi)$ is invertible, and $$\sup_{\phi\in\mathscr{D}}\{\|\sigma(\phi)\|_{\rm HS}+\|\sigma^{-1}(\phi)\|_{\rm HS}\}\leq K.$$
\end{conditionA}

Furthermore, we can verify that Conditions \ref{C1} and \ref{C2} hold, which will be completed in Sections \ref{SEC:03} and \ref{SEC:04}, respectively. Then, following the proof of \cite[Theorem 3.1]{Butkovsky2020}, 
we can prove the main result of this paper by Theorem \ref{Prop:0603:01}.
\begin{theorem}\label{THM:0603:01}
   Suppose that Assumptions \ref{A1} and \ref{A3} hold, and that either condition (i) or condition (ii) below is satisfied:
    \begin{enumerate}[(i)]
        \item \( \lim _{\|\phi\|_{\infty} \rightarrow \infty} V(\phi)=+\infty \).
        \item \( \lim _{|\phi(0)| \rightarrow \infty} V(\phi)=+\infty \). Assume additionally that there exists \( C>0 \) such that for any \( \phi \in \mathscr{D} \),
{\[
2\langle \phi(0), b(\phi)\rangle_{+} +\|\sigma(\phi)\|_{\rm{HS}}^2+|\gamma(\phi_{-})|^2\leq K\left(1+|\phi(0)|^2\right).
\] }
    \end{enumerate}
Then SFDEwJ \eqref{EQ:0603:01}  has a unique invariant measure $\pi$. Further, $P_t(\xi,\cdot)$, $t\ge \tau$, converges to $\pi$ in the Wasserstein metric $W_{d \wedge 1}$ and the rate of convergence is given by \eqref{EQ:0603:05}.
\end{theorem}

\section{Verification of Condition \ref{C1}}\label{SEC:03}
\subsection{Verification of Condition \ref{C1}.2}
 
Define a generalized coupled process $(X,Y)$ by
\begin{equation}\label{EQ:0527:05}
    \d X(t) = b(X_{t})\d t+\sigma(X_{t}) \d W(t)+\int_{Z}\gamma(X_{t-})c(z)\widetilde{N}(\d t,\d z),
\end{equation}
and 
\begin{equation}\label{EQ:0527:06}
    \d Y(t) = b(Y_{t})\d t+\sigma(Y_{t}) \d W(t)+\int_{Z}\gamma(Y_{t-})c(z)\widetilde{N}(\d t,\d z)+\lambda(X(t)-Y(t))\d t,
\end{equation}
with initial data $(X_0, Y_0)=(\xi, \eta)$. We shall provide a key estimate for this process as follows, which essentially guaranties that Condition \ref{C1}.2 holds.

\begin{lemma}\label{eq:X-Y}
    Suppose that Assumption \ref{A1}  holds. Given \(\alpha>0\), choose \(\lambda\) in \eqref{EQ:0527:06} satisfying \(2\lambda\geq \alpha+\bigl(1+\e^{\alpha\tau}\bigr)(K+K^2),\) where $K$ is given in Assumption \ref{A1}, then
    \begin{equation*}
        \E \|X_t-Y_t\|_{\infty}^2 \leq C\|\xi-\eta\|_{\infty}^2\e^{-\alpha t}, \quad  \text{for any} \ t\geq 0.
    \end{equation*}
\end{lemma}
\begin{proof}
    Due to \eqref{EQ:0527:05} and \eqref{EQ:0527:06}, we have 
    \begin{equation*}
  \begin{split}
            \d (X(t)-Y(t)) &= (b(X_{t})-b(Y_{t}))\d t+(\sigma(X_{t})- \sigma(Y_{t}))\d W(t)\\
            &\quad+\int_{Z}\big(\gamma(X_{t-})
            -\gamma(Y_{t-})\big)c(z)\widetilde{N}(\d t,\d z) -\lambda(X(t)-Y(t))\d t.
  \end{split}
    \end{equation*}
    By applying It\^o's formula to $|X(t)-Y(t)|^2$ and using Assumption \ref{A1}, it follows that 
    \begin{align}\label{Eq:newel}
         \nonumber   \d &|X(t)-Y(t)|^2\\
       \nonumber   &= \left(2\langle X(t)-Y(t),b(X_{t})-b(Y_{t})\rangle+\|\sigma(X_{t})-\sigma(Y_{t})\|_{\rm{HS}}^2 \right) \d t\\
      \nonumber    &\quad -2\lambda|X(t)-Y(t)|^2\d t+2\langle X(t)-Y(t),(\sigma(X_{t})-\sigma(Y_{t}))\d W(t)\rangle\\
      \nonumber    &\quad+ \int_{Z} \left(|X(t-)-Y(t-)+(\gamma(X_{t-})-\gamma(Y_{t-}))c(z)|^2 -|X(t-)-Y(t-)|^2\right)\widetilde{N}(\d t,\d z)\\
         &\quad +\int_{Z} \big(|X(t-)-Y(t-)+(\gamma(X_{t-})-\gamma(Y_{t-}))c(z)|^2 -|X(t-)-Y(t-)|^2\\
       \nonumber   &\qquad\qquad -2\langle X(t)-Y(t),(\gamma(X_{t-})-\gamma(Y_{t-}))c(z)\rangle\big)\nu(\d z)\d t\\
      \nonumber    &= \bigg(2\langle X(t)-Y(t),b(X_{t})-b(Y_{t})\rangle+\|\sigma(X_{t})-\sigma(Y_{t})\|_{\rm{HS}}^2  \\
       \nonumber   &\qquad +|\gamma(X_{t-})-\gamma(Y_{t-})|^2\int_{Z}|c(z)|^2\nu(\d z)\bigg) \d t -2\lambda|X(t)-Y(t)|^2\d t+\mathrm{d}M_1(t)+\mathrm{d}M_2(t)\\
      \nonumber    &\leq \big((K^{\prime}-2\lambda)|X(t)-Y(t)|^2+K^{\prime}\int_{-\tau}^0|X(t+\theta)-Y(t+\theta)|^2\mu(\d \theta)\big)\d t +\d M_1(t)+\d M_2(t),
        \end{align}
        where  $K^{\prime}=K+K^2$,
        \[
        \d M_1(t):=2\langle X(t)-Y(t),(\sigma(X_{t})-\sigma(Y_{t}))\d W(t)\rangle,
        \]
        and
        \[
        \d M_2(t):=\int_{Z} \left(|X(t-)-Y(t-)+(\gamma(X_{t-})-\gamma(Y_{t-}))c(z)|^2 -|X(t-)-Y(t-)|^2\right)\widetilde{N}(\d t,\d z).
        \]
    % Then it follows from  that, 
    % \begin{equation*}
    %     \begin{split}
    %         \d |X(t)-Y(t)|^2 &\leq K\left(|X(t)-Y(t)|^2+\int_{-\infty}^{0}|X(t+\theta)-Y(t+\theta)|^2\mu(\d \theta)\right)\d t\\
    %         &\quad -2\lambda |X(t)-Y(t)|^2\d t+2\langle X(t)-Y(t),(\sigma(X_{t-})-\sigma(Y_{t-}))\d W(t)\rangle \\
    %         &\quad +\int_{Z}\big(|X(t-)-Y(t-)+(\gamma(X_{t-})-\gamma(Y_{t-}))c(z)|^2 \\
    %         &\qquad\qquad-|X(t-)-Y(t-)|^2\big)\widetilde{N}(\d t,\d z)
    %         \\
    %         &= :
    %     \end{split}
    % \end{equation*}
     To proceed, define the stopping time $T_\varepsilon := \inf\left\{ t \ge 0 : |X(t) - Y(t)| \ge \varepsilon^{-1} \right\}$. Given $\alpha > 0 $, it follows from  \eqref{Eq:newel} that
     \begin{equation*}
         \begin{split}
             \mathbb{E}&\left(\e^{\alpha (t\wedge T_\varepsilon)}|X(t\wedge T_\varepsilon)-Y(t\wedge T_\varepsilon)|^2\right)\\
             &\leq |\xi(0)-\eta(0)|^2+(\alpha+K^{\prime}-2\lambda)\mathbb{E}\int_0^{t\wedge T_\varepsilon}\e^{\alpha s}|X(s)-Y(s)|^2\d s\\
             &\quad+K^{\prime}\mathbb{E}\int_0^{t\wedge T_\varepsilon}\int_{-\tau}^0\e^{\alpha s}|X(s+\theta)-Y(s+\theta)|^2\mu(\d \theta)\big)\d s.
         \end{split}
     \end{equation*}
By Fubini's theorem, we get
    \begin{align*}
\int_{0}^{t\wedge T_\varepsilon}&\int_{-\tau}^{0} \e^{\alpha s}\left|X(s+\theta)-Y(s+\theta)\right|^{2}\mu(\mathrm{d}\theta)\mathrm{d}s\\
&=\int_{-\tau}^{0}\int_{0}^{t\wedge T_\varepsilon} \e^{\alpha s}\left|X(s+\theta)-Y(s+\theta)\right|^{2}\mathrm{d}s\,\mu(\mathrm{d}\theta) \\
&=\int_{-\tau}^{0}\int_{\theta}^{t\wedge T_\varepsilon+\theta} \e^{-\alpha\theta} \e^{\alpha u}\left|X(u)-Y(u)\right|^{2}\mathrm{d}u\,\mu(\mathrm{d}\theta) \\
&\le \e^{\alpha\tau}\int_{-\tau}^{t\wedge T_\varepsilon} \e^{\alpha s}\left|X(s)-Y(s)\right|^{2}\mathrm{d}s \\
&\le \tau \e^{\alpha\tau}\|\xi-\eta\|_{\infty}^{2}+\e^{\alpha\tau}\int_{0}^{t\wedge T_\varepsilon} \e^{\alpha s}\left|X(s)-Y(s)\right|^{2}\mathrm{d}s.
\end{align*}
Hence,  
\begin{equation*}
    \begin{aligned}
        \mathbb{E}&\left(\e^{\alpha (t\wedge T_\varepsilon)}|X(t\wedge T_\varepsilon)-Y(t\wedge T_\varepsilon)|^2\right)\\
        &\leq C\|\xi-\eta\|_{\infty}^2+(\alpha+(1+\e^{\alpha\tau})K^{\prime}-2\lambda)\mathbb{E}\int_0^{t\wedge T_\varepsilon}\e^{\alpha s}|X(s)-Y(s)|^2\d s.
    \end{aligned}        
\end{equation*}
For any given $\alpha>0$, choose $\lambda=\lambda_0>0$ such that $\alpha+(1+\e^{\alpha\tau})K^{\prime}-2\lambda\leq 0$. Then, we obtain
\begin{equation}\label{Eq:zhognj}
   \mathbb{E}\left(\e^{\alpha (t\wedge T_\varepsilon)}|X(t\wedge T_\varepsilon)-Y(t\wedge T_\varepsilon)|^2\right)\leq  C\|\xi-\eta\|_{\infty}^2
\end{equation}
for any $t\geq 0$. We next prove that $T_{\varepsilon}\to \infty$ a.s., as $\varepsilon\to 0$. By Chebyshev's inequality, we arrive at
\[
\begin{aligned}
\mathbb{P}(T_\varepsilon \ge t)
&\le \mathbb{P}\left( |X(t\wedge T_\varepsilon) - Y(t\wedge T_\varepsilon)| \ge \varepsilon^{-1} \right) \\
&\le \varepsilon^2\,\mathbb{E}\left(|X(t\wedge T_\varepsilon) - Y(t\wedge T_\varepsilon)|^2\right) \\
&\le \varepsilon^2\,\mathbb{E}\left(\e^{\alpha(t\wedge T_\varepsilon)}|X(t\wedge T_\varepsilon) - Y(t\wedge T_\varepsilon)|^2\right) \\
&\le C\varepsilon^2\|\xi-\eta\|_\infty^2,
\end{aligned}
\]
for any $t\geq 0$. Taking $\varepsilon\to 0$ yields $T_{\varepsilon}\to \infty$ a.s.  Furthermore, applying the monotone convergence theorem, we have
\begin{equation}\label{Eq:zhognjer}
    \mathbb{E}|X(t)-Y(t)|^2\leq C\|\xi-\eta\|_{\infty}^2\e^{-\alpha t},
\end{equation}
for any $t\geq 0$.
It is clear that \eqref{Eq:zhognjer} also holds for $t\in[-\tau,0]$. Hence, \eqref{Eq:zhognjer} is valid for all $t\geq -\tau$.

Now we turn to prove our main result. Similarly, we can define the stopping time $ T_{\varepsilon}^{\prime}:=\inf\{{t\geq 0: \|X_t-Y_t\|_{\infty}\geq \varepsilon^{-1}\}}\vee \tau$. Let $t(\varepsilon)=t\wedge T_{\varepsilon}^{\prime}$. By \eqref{Eq:newel} and $2\lambda\geq K^{\prime}$, we have
\begin{equation*}
    \begin{aligned}
        \begin{aligned}
\mathbb{E}&\|X_{t(\varepsilon)}-Y_{t(\varepsilon)}\|_{\infty}^2\\
&=\mathbb{E}\left(\sup_{t(\varepsilon)-\tau \le s \le t(\varepsilon)} |X(s)-Y(s)|^2\right) \\
&\le \mathbb{E}|X(t(\varepsilon)-\tau)-Y(t(\varepsilon)-\tau)|^2
+K^{\prime}\mathbb{E}\int_{t(\varepsilon)-\tau}^{t(\varepsilon)}\int_{-\tau}^{0}|X(u+\theta)-Y(u+\theta)|^2\mu(\mathrm{d}\theta)\mathrm{d}s \\
&\quad+\mathbb{E}\left(\sup_{t(\varepsilon)-\tau \le s \le t(\varepsilon)}\left(M_1(s)+M_2(s)\right)\right).
\end{aligned}
    \end{aligned}
\end{equation*} 
 Using BDG's inequality and Young's inequality implies
\begin{equation*}
    \begin{aligned}
\mathbb{E}&\left(\sup_{t(\varepsilon)-\tau\le s\le t(\varepsilon)} M_1(s)\right)\\
&\le 8\sqrt{2}\,\mathbb{E}\left(\int_{t(\varepsilon)-\tau}^{t(\varepsilon)} |X(s)-Y(s)|^2\ \|\sigma(X_s)-\sigma(Y_s)\|_{\rm HS}^2\,\mathrm{d}s\right)^{\frac12} \\
&\le \frac14\,\mathbb{E}\left(\sup_{t(\varepsilon)-\tau\le s\le t(\varepsilon)}|X(s)-Y(s)|^2\right)
+C\,\mathbb{E}\left(\int_{t(\varepsilon)-\tau}^{t(\varepsilon)}\|\sigma(X_s)-\sigma(Y_s)\|_{\rm HS}^2\,\mathrm{d}s\right).
\end{aligned}
\end{equation*}
Then, by BDG's inequality \cite[Theorem 48]{protter2004stochastic}, and the definition of stochastic integral w.r.t. the Poisson jump process, we obtain that
\begin{equation*}
{
\begin{aligned}
\mathbb{E}\Big(\sup_{t(\varepsilon)-\tau\le s\le t(\varepsilon)} M_2(s)\Big)&\le C\mathbb{E}\left(\int_{t(\varepsilon)-\tau}^{t(\varepsilon)}\int_{\mathbb{Z}}\left|2\langle X(s-)-Y(s-),\,(\gamma(X_{s-})-\gamma(Y_{s-})) c(z)\rangle
\right.\right.\\
&\qquad\qquad\qquad\left.\left.+\left|\gamma(X_{s-})-\gamma(Y_{s-})\right|^{2}|c(z)|^{2}\right|^{2}N(\mathrm{d}s,\mathrm{d}z)\right)^{\frac12} \\[-3pt]
&= C\mathbb{E}\Big(\sum_{\substack{t(\varepsilon)-\tau\le u\le t(\varepsilon)\\ u\in D_p\\ p(u)\in\mathbb{Z}}}\left|2\langle X(u-)-Y(u-),\,(\gamma(X_{u-})-\gamma(Y_{u-})) c(z)\rangle\right.\\[-3pt]
&\qquad\qquad\qquad\left.+\left|\gamma(X_{u-})-\gamma(Y_{u-})\right|^{2}|c(z)|^{2}\right|^{2}\Big)^{\frac12} \\[-3pt]
&\le C\mathbb{E}\Big(\sum_{\substack{t(\varepsilon)-\tau\le u\le t(\varepsilon)\\ u\in D_p\\ p(u)\in\mathbb{Z}}}\left|\langle X(u-)-Y(u-),\,(\gamma(X_{u-})-\gamma(Y_{u-}))c(z)\rangle\right|\Big) \\[-3pt]
&\quad + C\mathbb{E}\Big(\sum_{\substack{t(\varepsilon)-\tau\le u\le t(\varepsilon)\\ u\in D_p\\ p(u)\in\mathbb{Z}}}\left|\gamma(X_{u-})-\gamma(Y_{u-})\right|^{2}|c(z)|^{2}\Big)\\
&= C\mathbb{E}\Big(\sum_{\substack{t(\varepsilon)-\tau\le u\le t(\varepsilon)\\ u\in D_p\\ p(u)\in\mathbb{Z}}}\left|\langle X(u-)-Y(u-),\,(\gamma(X_{u-})-\gamma(Y_{u-}))c(z)\rangle\right|\Big) \\[-3pt]
&\quad + C\mathbb{E}\left(\int_{t(\varepsilon)-\tau}^{t(\varepsilon)}\int_{\mathbb{Z}}|\gamma(X_{s-})-\gamma(Y_{s-})|^2|c(z)|^2\,\nu(\mathrm{d}z)\mathrm{d}s\right).
\end{aligned}
}
\end{equation*}

For the first term, Young's inequality and H\"older's inequality  imply that
\begin{align*}
C&\,\mathbb{E}\Big(\sum_{\substack{t(\varepsilon)-\tau\le u\le t(\varepsilon)\\ u\in D_p\\ p(u)\in\mathbb{Z}}}\left|\langle X(u-)-Y(u-),\,(\gamma(X_{u-})-\gamma(Y_{u-})) c(z)\rangle\right|\Big) \\
&\le C\,\mathbb{E}\int_{t(\varepsilon)-\tau}^{t(\varepsilon)}\int_{\mathbb{Z}}|X(s-)-Y(s-)|\cdot|\gamma(X_{s-})-\gamma(Y_{s-})|\cdot|c(z)|\,N(\mathrm{d}s,\mathrm{d}z) \\
&= C\,\mathbb{E}\int_{t(\varepsilon)-\tau}^{t(\varepsilon)}\int_{\mathbb{Z}}|X(s-)-Y(s-)|\cdot|\gamma(X_{s-})-\gamma(Y_{s-})|\cdot|c(z)|\,\nu(\mathrm{d}z)\mathrm{d}s \\
&\le C\,\mathbb{E}\left(\sup_{t(\varepsilon)-\tau\le s\le t(\varepsilon)}|X(s-)-Y(s-)|\int_{t-\tau}^{t}\int_{\mathbb{Z}}|\gamma(X_{s-})-\gamma(Y_{s-})|\cdot|c(z)|\,\nu(\mathrm{d}z)\mathrm{d}s\right) \\
&\le \frac14\,\mathbb{E}\left(\sup_{t(\varepsilon)-\tau\le s\le t(\varepsilon)}|X(s)-Y(s)|^2\right) \\
&\quad +C\,\mathbb{E}\left(\int_{t(\varepsilon)-\tau}^{t(\varepsilon)}\int_{\mathbb{Z}}|\gamma(X_{s-})-\gamma(Y_{s-})|^2|c(z)|^2\,\nu(\mathrm{d}z)\mathrm{d}s\right).
\end{align*}
Combining the above calculations yields 
 \begin{align*}
\mathbb{E}&\|X_{t(\varepsilon)}-Y_{t(\varepsilon)}\|_{\infty}^2\\
&\le \mathbb{E}|X(t(\varepsilon)-\tau)-Y(t(\varepsilon)-\tau)|^2
+K^{\prime}\mathbb{E}\int_{t(\varepsilon)-\tau}^{t(\varepsilon)}\int_{-\tau}^{0}|X(u+\theta)-Y(u+\theta)|^2\mu(\mathrm{d}\theta)\mathrm{d}s \\
&\quad+\frac12\,\mathbb{E}\left(\sup_{t(\varepsilon)-\tau\le s\le t(\varepsilon)}|X(s)-Y(s)|^2\right) +C\,\mathbb{E}\left(\int_{t(\varepsilon)-\tau}^{t(\varepsilon)}\|\sigma(X_s)-\sigma(Y_s)\|_{\rm HS}^2\,\mathrm{d}s\right)\\
&\quad +C\,\mathbb{E}\left(\int_{t(\varepsilon)-\tau}^{t(\varepsilon)}\int_{\mathbb{Z}}|\gamma(X_{s-})-\gamma(Y_{s-})|^2|c(z)|^2\,\nu(\mathrm{d}z)\mathrm{d}s\right)\\
&\le \frac12\,\mathbb{E}\left(\sup_{t(\varepsilon)-\tau\le s\le t(\varepsilon)}|X(s)-Y(s)|^2\right)+\mathbb{E}|X(t(\varepsilon)-\tau)-Y(t(\varepsilon)-\tau)|^2
 \\
&\quad +C\mathbb{E}\int_{t(\varepsilon)-\tau}^{t(\varepsilon)}|X(s)-Y(s)|^2\mathrm{d}s+C\mathbb{E}\int_{t(\varepsilon)-\tau}^{t(\varepsilon)}\int_{-\tau}^{0}|X(u+\theta)-Y(u+\theta)|^2\mu(\mathrm{d}\theta)\mathrm{d}s.
\end{align*}
By Fubini's theorem, we get
\begin{equation*}
    \begin{aligned}
\int_{t(\varepsilon)-\tau}^{t(\varepsilon)}\int_{-\tau}^{0} \left|X(s+\theta)-Y(s+\theta)\right|^{2}\mu(\mathrm{d}\theta)\mathrm{d}s
&=\int_{-\tau}^{0}\int_{t(\varepsilon)-\tau}^{t(\varepsilon)} \left|X(s+\theta)-Y(s+\theta)\right|^{2}\mathrm{d}s\,\mu(\mathrm{d}\theta) \\
&=\int_{-\tau}^{0}\int_{t(\varepsilon)-\tau+\theta}^{t(\varepsilon)+\theta} \left|X(u)-Y(u)\right|^{2}\mathrm{d}u\,\mu(\mathrm{d}\theta) \\
&\le \int_{t(\varepsilon)-2\tau}^{t(\varepsilon)} \left|X(s)-Y(s)\right|^{2}\mathrm{d}s.
\end{aligned}
\end{equation*}
 Hence, for any $t\geq \tau$, it follows from \eqref{Eq:zhognjer} that
\begin{equation*}
    \begin{aligned}
    \mathbb{E}\|X_{t(\varepsilon)}-Y_{t(\varepsilon)}\|_{\infty}^2
&\le 2\mathbb{E}|X(t(\varepsilon)-\tau)-Y(t(\varepsilon)-\tau)|^2+C\mathbb{E}\int_{t(\varepsilon)-\tau}^{t(\varepsilon)}|X(s)-Y(s)|^2\mathrm{d}s\\
&\quad +C\mathbb{E}\int_{t(\varepsilon)-2\tau}^{t(\varepsilon)}|X(s)-Y(s)|^2\mathrm{d}s\\
&\leq C\|\xi-\eta\|_{\infty}^2\e^{-\alpha t},
\end{aligned} 
\end{equation*}
where we have used the fact that $t(\varepsilon)-2\tau\geq -\tau$. As in \eqref{Eq:zhognjer}, we have
\begin{equation}\label{Eq:shuijia}
    \mathbb{E}\|X_{t}-Y_{t}\|_{\infty}^2\leq C\|\xi-\eta\|_{\infty}^2\e^{-\alpha t}.
\end{equation}
For $0 \le t \le \tau$, note that
\[
\mathbb{E}\|X_t - Y_t\|_\infty^2
\le \|\xi - \eta\|_\infty^2
+ \mathbb{E}\left(\sup_{0 \le s \le t} |X(s) - Y(s)|^2\right).
\]
Using a similar argument for deriving \eqref{Eq:shuijia}, we deduce that
\[
\mathbb{E}\|X_t - Y_t\|_\infty^2
\le C\|\xi-\eta\|_\infty^2 \e^{-\alpha t},
\quad 0 \le t \le \tau.
\]
Hence, we obtain the desired result. The proof is complete.
\end{proof}

\subsection{Verification of Condition \ref{C1}.1}
	Recall that for a pair of probability measures \( \mu \ll \nu \) over a measurable space \( (X, \mathcal{X}) \), the Kullback-Leibler (KL-) divergence of \( \mu \) from \( \nu \) is defined by
	\[
	D_{\mathrm{KL}}(\mu \| \nu):=\int_{X} \log \frac{\d \mu}{\d \nu} \d \mu=\int_{X} \frac{\d \mu}{\d \nu} \log \left(\frac{\d \mu}{\d \nu}\right) \d \nu.
	\]
	If a measure \( \mu \) is not absolutely continuous with respect to \( \nu \), then, for convenience, we put \( D_{\mathrm{KL}}(\mu \| \nu):=+\infty \). KL-divergence is a stronger measure of the difference between probability distributions than the total variation distance, i.e.,
    \[
    d_{\mathrm{TV}}(\mu, \nu)\leq \sqrt{\frac{1}{2} D_{K L}(\mu \| \nu)}.
    \]
	
	Now consider a \( n \)-dimensional \( (n \in \mathbb{N}) \) It\^o-L\'evy process \( \left(\xi(t)\right)_{t \geq 0} \) with \( \xi_{0}=0 \) and
	\[
	\d \xi(t)=\beta(t) \d t+\d W(t)+\int_{Z}z \widetilde{N}(\d t,\d z), \quad t \geq 0,
	\]
	where \( W(t) \) is a Wiener process in \( \mathbb{R}^{n} \), $N(\d t,\d z)$ is a Poisson random measure, and \( \left(\beta(t)\right)_{t \geq 0} \) is a progressively measurable process. Let \(\mu_{\xi}\) be the law of the process \(\xi\) in \(D([0, \infty);\mathbb{R}^n)\), \(\mu_W\) the law of $W$ in \(C([0, \infty); \mathbb{R}^n)\), and \(\mu_L\) the law of $L(\cdot):=\int_0^{\cdot}\int_{Z}z \widetilde{N}(\d s,\d z)$ in \(D([0, \infty); \mathbb{R}^n)\). Using the Girsanov theorem for It\^o–L\'evy processes (see, e.g., \cite[Theorem 1.31]{Oksendal}) and an argument similar to \cite[Theorem A.2]{Butkovsky2020}, we get the following result.  
\begin{theorem}\label{Thm:mhsh:01}

		\[
		\begin{aligned}
        D_{\mathrm{KL}}&\left(\mu_{\xi} \| \mu_{W}\otimes\mu_{L}\right) \\&\leq \mathbb{E} \left(  \frac{1}{2} \int_{0}^{\infty} \theta_{0}^{2}(t) \mathrm{d} t+\int_{0}^{\infty} \int_{Z}\left[\left(1-\theta_{1}(t, z)\right)\log \left(1-\theta_{1}(t, z)\right)+\theta_{1}(t, z)\right] \nu(\mathrm{d} z) \mathrm{d} t\right), 
		\end{aligned}
		\]
		where  \( \theta_{0}(t)=\theta_{0}(t, \omega) \in \mathbb{R}^{n} \) and \( \theta_{1}(t, z)= \) \( \theta_{1}(t, z, \omega) \in \mathbb{R} \) are predictable processes satisfying
        \begin{equation}\label{Eq:1122t}
            \theta_{0}(t)+\int_{Z}z \theta_{1}(t, z) \nu(\mathrm{d} z)=\beta(t) \text { for a.a. }(t, \omega) \in[0, T] \times \Omega.
        \end{equation}
\end{theorem}

 Building on the above results, we now proceed to verify Condition \ref{C1}.1. Let
\[
\d W^{(\lambda)}(t):=\lambda\sigma^{-1}(Y_t)(X(t)-Y(t))\d t+\d W(t).
\]
Then $Y(t)$ satisfies the following equation
\[
\d Y(t)=b(Y_t)\d t+\sigma(Y_t)\d W^{(\lambda)}(t)+\int_{Z}\gamma(Y_{t-})c(z)\widetilde{N}(\d t,\d z).
\]
Recall that for any \( t \geq 0 \), the strong solution \( ( X_t^{\xi})_{t\geq 0} \) to \eqref{EQ:0527:05}  is an image of the driving noise under a measurable mapping
\[
\Phi_{t}^{\xi}: C\left([0, t], \mathbb{R}^{n}\right)\times D([0, t]; \mathbb{R}^n) \rightarrow \mathbb{R}^n.
\]
 In other words, \( X^{\xi}_t=\Phi_{t}^{\xi}(W_{[0, t]}, L_{[0, t]}) \). By the Girsanov theorem (\cite[Theorem 7.4]{liptser2013}), it follows that  \( \operatorname{Law}(W_{[0, t]}^{(\lambda)}) \) is absolutely continuous with respect to \( \operatorname{Law}(W_{[0, t]}) \). Therefore, by the uniqueness of the solution, we have \( Y^{ \eta}_t=\Phi_{t}^{\eta}(W_{[0, t]}^{(\lambda)},L_{[0, t]}) \). Choose $\theta_0(t)=\lambda_0\sigma^{-1}(Y_t)(X(t)-Y(t))$ and $\theta_1(t,z)=0$ in \eqref{Eq:1122t}, where $\lambda_0$ is given in Lemma \ref{eq:X-Y}. Then, by Theorem \ref{Thm:mhsh:01}, Lemma \ref{eq:X-Y}, and Assumption \ref{A3}, we get
		\begin{align*}
			 d_{\mathrm{TV}}\left(\operatorname{Law}(Y_{t}^{ \eta}), P_{t}(\eta, \cdot)\right) 
			&=  d_{\mathrm{TV}}\left(\operatorname{Law}(\Phi_t^{\eta}(W_{[0, t]}^{ (\lambda)},L_{[0, t]})), \operatorname{Law}(\Phi_t^{\eta}(W_{[0, t]}, L_{[0, t]}))\right) \\
			&\leqslant  d_{\mathrm{TV}}\left(\operatorname{Law}(W_{[0, t]}^{(\lambda) },L_{[0, t]}), \operatorname{Law}(W_{[0, t]}, L_{[0, t]})\right) \\
            &\leqslant \sqrt{\frac{1}{2} D_{K L}\left(\operatorname{Law}(W_{[0, t]}^{(\lambda) },L_{[0, t]}) \|\operatorname{Law}(W_{[0, t]}, L_{[0, t]})\right)}\\
			&\leqslant  \frac{\lambda_0}{2} \sqrt{\mathbb E\left(\int_0^t\|\sigma^{-1}(Y_s)\|_{\mathrm{HS}}^2|X(s)-Y(s)|^2\d s\right)}\\
            &\leq \frac{\lambda_0 K\sqrt{C}}{2\sqrt{\alpha}}\|\xi-\eta\|_{\infty},
		\end{align*}
    which implies that Condition \ref{C1}.1 holds with $C(t)=\lambda_0 K\sqrt{C}/(2\sqrt{\alpha})$.
	
\section{Verification of Condition \ref{C2}}\label{SEC:04}
We now proceed to verify Condition \ref{C2}. 
%Let \(C((-\infty, 0]; \mathbb{R}^n)\) be all continuous functions from \((-\infty, 0]\) to \(\mathbb{R}^n\). Define \(\mathscr{C}_r\) analogously to \(\mathscr{D}\).  
In what follows, we derive a support theorem for the continuous stochastic process \(\widetilde{X}(t)\), governed by 
\begin{equation}\label{eq:aux1}
    \d \widetilde{X}(t)=b(\widetilde{X}_{t})\d t+\sigma(\widetilde{X}_{t})\d W(t)-\int_{Z}\gamma(\widetilde{X}_{t})c(z)\nu(\d z)\d t,\quad X_0=\xi.
\end{equation}
Then, we can obtain the following result using the arguments analogous to those in Lemma 3.8 of \cite{Hairer2011PTRF}.
\begin{lemma}\label{lem:widetildeX}
Suppose that Assumptions \ref{A1} and \ref{A3} hold. Let \(B_{c}:=\{\phi\in\mathscr{D}\mid {\|\phi\|_{\infty}}\leq c\}\) for some constant $c>0$. Then for any $R, \delta>0$, there exists a time threshold $t_{R,\delta}>0$ such that
    \begin{equation}
        \inf_{\xi\in {B}_R}\P(\widetilde{X}^{\xi}_{t}\in B_{\delta})>0, \quad t\geq t_{R, \delta}.
    \end{equation}
\end{lemma}

Based on the above result, we can easily get the corresponding support theorem for the stochastic process ${X}^{\xi}_{t}$.
\begin{lemma}\label{Lem:0809}
    Suppose that Assumptions \ref{A1} and \ref{A3} hold. For any $R, \delta>0$, we have
    \begin{equation}
        \inf_{\xi\in  {B}_R}\P({X}^{\xi}_{t}\in B_{\delta})>0, \quad t\geq t_{R, \delta},
    \end{equation}
    where $t_{R, \delta}$ and $ {B}_{R}$ are defined in Lemma \ref{lem:widetildeX}.
\end{lemma}

\begin{proof}
    Define a stopping time $\varsigma:=\inf\{t\geq0: X(t)\neq X(t-)\}$. Note that the conditional distribution of ${X}^{\xi}_{t}$ on the set $\{\varsigma>t\}$ is identical with the distribution of $\widetilde{X}^{\xi}_{t}$.   
    % Note that ${X}^{\xi}_{t}$ and $\widetilde{X}^{\xi}_{t}$ have the same distribution on the set $\{\tau_1>t\}$. 
    Hence, for any $\xi\in {B}_R$ and $t\geq t_{R, \delta}$, we have
    \begin{equation*}
        \begin{split}
            \P({X}^{\xi}_{t}\in B_{\delta})
            &\geq \P({X}^{\xi}_{t}\in B_{\delta}, \varsigma>t)\\
            &=\P({X}^{\xi}_{t}\in B_{\delta}|\varsigma>t)\P(\varsigma>t)\\
            &= \P({X}^{\xi}_{t}\in B_{\delta}|\varsigma>t)\e^{-\nu(Z)t}\\
            &=\P(\widetilde{X}^{\xi}_{t}\in  {B}_{\delta})\e^{-\nu(Z)t}.
        \end{split}
    \end{equation*}
 Taking the infimum over $\xi\in {B}_R$ on both sides of the above inequality, we derive our desired assertion by Lemma \ref{lem:widetildeX}.
\end{proof}
%There is the position to give\\
\noindent\textbf{Verification of Condition \ref{C2}}: 
\begin{enumerate}[(1)]
    \item  If \( \lim _{\|\phi\|_{\infty} \rightarrow \infty} V(\phi)=+\infty \), there exists a positive constant $\widetilde{M}$ such that $B=\{V\leq M\}=\{\phi\in\mathscr{D}\mid {\|\phi\|_{\infty}}\leq \widetilde{M}\}$. Given \( \varepsilon>0 \), we take \( D=B_{\varepsilon/2}\). Then, clearly, Condition \ref{C2}.2 holds. Taking $R=\widetilde{M}$ and $\delta=\varepsilon/2$ in Lemma \ref{Lem:0809} implies Condition \ref{C2}.1. 
    \item   If \( \lim _{|\phi(0)| \rightarrow \infty} V(\phi)=+\infty \), there exists a positive constant $\widehat{M}$ such that $B=\{V\leq M\}=\{\phi\in\mathscr{D}\mid |\phi(0)|\leq \widehat{M}\}=:\widehat{B}$. Clearly, Condition \ref{C2}.2 holds with \( D=B_{\varepsilon/2}\). We only need to verify that Condition \ref{C2}.1 holds for $\widehat{B}$.
By It\^o's formula, we obtain
    \begin{align*}
        |X(t)|^2&=|\xi(0)|^2+\int_0^t\left(2\langle X(s), b(X_s)\rangle+\|\sigma(X_s)\|^2_{\mathrm{HS}}+\int_Z |\gamma(X_{s-})c(z)|^2\nu(\d z)\right)\d s\\
        & \quad+\int_0^t 2\langle X(s), \sigma(X_s)\d W(s)\rangle\\
        &\quad+\int_0^t \int_Z\left(|X(s-)+\gamma(X_{s-})c(z)|^2-|X(s-)|^2\right)\widetilde{N}(\d s, \d z).
    \end{align*}
For any $\delta>0$, let $T_{\delta}:=\inf\{t\geq0:|X(t)|\geq \delta^{-1}\}$. By \cite[Theorem 1]{Novikov1975} and Assumption \ref{A1}, we have for any $t\in[0,\tau]$,
\begin{align*}
    \mathbb{E}&\left(\sup_{0\leq s\leq t\wedge  T_{\delta}}|X(s)|^2\right)\\
    &\leq |\xi(0)|^2+\mathbb E\int_0^{t\wedge T_{\delta}}\left(2\langle X(s), b(X_s)\rangle_++\|\sigma(X_s)\|^2_{\mathrm{HS}}+\int_Z |\gamma(X_{s-})c(z)|^2\nu(\d z)\right)\d s\\
    &\quad +C\mathbb E\left(\int_0^{t\wedge  T_{\delta}} |X(s)|^2\|\sigma(X_s)\|^2_{\mathrm{HS}}\d s\right)^{\frac{1}{2}}+C\nu(Z)\E\bigg(\int_0^{t\wedge  T_{\delta}}|X(s)|^2\d s\bigg)\\
    &\quad +C\E\bigg(\int_0^{t\wedge   T_{\delta}}\int_{Z}|\gamma(X_{s-})c(z)|^2\nu(\d z)\d s\bigg)\\
    &\leq |\xi(0)|^2+ C\mathbb E\int_0^{t\wedge T_{\delta}}\left(1+|X(s)|^2\right)\d s +\frac{1}{2}\mathbb E\left(\sup_{0\leq s\leq t\wedge T_{\delta}}|X(s)|^2\right)\\
    &\leq C+|\xi(0)|^2+\frac{1}{2}\mathbb E\left(\sup_{0\leq s\leq t\wedge  T_{\delta}}|X(s)|^2\right)+C\int_0^{t}\mathbb E\left(\sup_{0\leq u\leq s\wedge T_{\delta}}|X(u)|^2\right) \d s,
\end{align*}
  which implies that
   \[
   \mathbb E\left(\sup_{0\leq s\leq t\wedge T_{\delta}}|X(s)|^2\right)\leq C(1+|\xi(0)|^2)+C\int_0^{t}\mathbb E\left(\sup_{0\leq u\leq s\wedge T_{\delta}}|X(u)|^2\right) \d s.
   \]
Therefore, using Gr\"onwall's inequality, one can obtain
\[
\mathbb E\left(\sup_{0\leq s\leq \tau\wedge T_{\delta}}|X(s)|^2\right)\leq C(1+|\xi(0)|^2)\e^{C\tau}.
\]
Letting $\delta\to 0$, by means of Fatou's lemma, we finally derive
\[
\mathbb E\left(\sup_{0\leq s\leq \tau}|X(s)|^2\right)\leq C(1+|\xi(0)|^2)\e^{C\tau}.
\]
Note that 
    $\|X_\tau\|_{\infty}=\sup_{-\tau\leq \theta\leq0} |X(\tau+\theta)|=\sup_{0\leq u\leq \tau}|X(u)|$.
Thus, by Chebyshev's inequality, there exists a  constant $L=L(\widehat{M})>0$ large enough such that
\[
\inf_{\xi\in \widehat{B}}\mathbb P(\|X^\xi_\tau\|_{\infty}\leq L)>\frac{1}{2}.
\]
Taking $R=L$ and $\delta=\varepsilon/2$ in Lemma \ref{Lem:0809} and choosing $t_0\geq \tau+t_{L,\varepsilon/2}$, this implies
\begin{align*}
\inf_{\xi\in \widehat{B}}\mathbb P\big(\|X^\xi_{t_0}\|_{\infty}\leq\frac{\varepsilon}{2}\big)
&\geq \inf_{\xi\in \widehat{B}}\mathbb P\big(\|X^\xi_{t_0}\|_{\infty}\leq\frac{\varepsilon}{2}\mid \|X^\xi_\tau\|\leq L\big)\inf_{\xi\in \widehat{B}}\mathbb P(\|X^\xi_\tau\|\leq L)\\
&> \frac{1}{2}\inf_{\xi:\|\xi\|_\infty\leq L}\mathbb P\big(\|X^\xi_{t_0-\tau}\|_{\infty}\leq \frac{\varepsilon}{2}\big)>0.
\end{align*}
 The proof is complete.
\end{enumerate}

\bibliographystyle{plain}
%使得没有引用的参考文献能被一并显示
%\nocite{*}

\bibliography{Ergodicity_SFDEwJ}
															
\end{document}